\documentclass[a4paper]{article}
\usepackage{mathrsfs, amsmath, amsxtra, amssymb, latexsym, amscd, amsthm}

\newtheorem{thm}{Theorem}[section]

\newtheorem{prop}{Proposition}[section]
\theoremstyle{definition}

\theoremstyle{remark}

\numberwithin{equation}{section}

\begin{document}

\title{It\^o differential representation of singular stochastic Volterra integral equations}%
\author{Nguyen Tien Dung\footnote{Email: dung\_nguyentien10@yahoo.com}
}

\date{May 9, 2020}
\maketitle
\begin{abstract}
In this paper we obtain an It\^o differential representation for a class of singular stochastic Volterra integral equations. As an application, we investigate the rate of convergence in the small time central limit theorem for the solution.
\end{abstract}

\section{Introduction}Let $(B_t)_{t\in [0,T]}$ be a standard Brownian motion defined on a complete probability space $(\Omega,\mathcal{F},\mathbb{F},P)$, where $\mathbb{F}=(\mathcal{F}_t)_{t\in [0,T]}$ is a natural filtration generated by $B.$ It is known that stochastic Volterra integral equations of the form
\begin{equation}\label{fko3k1}
X_t=x_0+\int_0^tk(t,s)b(s,X_s)ds+\int_0^tg(t,s)\sigma(s,X_s)dB_s,\,\,\,t\in[0,T]
\end{equation}
play an important role in various fields such as mathematical finance, physics, biology, etc. Because of its applications, the class of stochastic Volterra integral equations has been investigated in several works. Among others, we mention \cite{Cochran1995,Wu2012} and references therein for the existence and uniqueness of solutions, \cite{Ferreyra2000} for comparison theorems, \cite{Mao2006,Son2018} for stability results, \cite{Jovanovi2002} for perturbed equations, \cite{Maleknejad2012} for numerical solutions.

One of main difficulties when studying the properties of (\ref{fko3k1}) and related problems is that we can not directly apply It\^o differential formula to (\ref{fko3k1}). In the case of the equations with regular kernels, we can avoid this difficult by rewriting (\ref{fko3k1}) in its It\^o differential form. In fact, we have
$$\int_0^tk(t,s)b(s,X_s)ds=\int_0^tk(s,s)b(s,X_s)ds+\int_0^t\bigg(\int_0^s \frac{\partial}{\partial s}k(s,u)b(u,X_u)du\bigg)ds,$$
$$\int_0^tg(t,s)\sigma(s,X_s)dB_s=\int_0^tg(s,s)\sigma(s,X_s)dB_s+\int_0^t\bigg(\int_0^s\frac{\partial}{\partial s}g(s,u)\sigma(u,X_u)dB_u\bigg)ds$$
and we obtain the following It\^o differential representation
$$dX_t=\big(k(t,t)b(t,X_t)+K(t)+G(t)\big)dt+g(t,t)\sigma(t,X_t)dB_t,\,\,\,t\in[0,T],$$
where $X_0=x_0\in \mathbb{R},$ $K(t):=\int_0^t \frac{\partial}{\partial t}k(t,u)b(u,X_u)du$ and $G(t):=\int_0^t\frac{\partial}{\partial t}g(t,u)\sigma(u,X_u)dB_u.$

Generally, this nice representation does not hold anymore if the kernels are singular, except we imposed additional conditions. The case, where  the kernel $g(t,s)$ is singular, has been discussed in our recent work \cite{ntdung2018}. We have to use the techniques of Malliavin calculus to obtain an It\^o-type differential transformation involving Malliavin derivatives and Skorokhod integrals, it is open to obtain an It\^o differential representation for the solution in this case. The aim of the present paper is to study the case, where  the kernel $k(t,s)$ is singular. 


In the whole paper, we require the assumption

\noindent$(A_1)$ The kernel $k(t,s)$ is defined on $\{0\leq s<t\leq T\}$ and has the following properties, for some $c>0,$

(i) $k(t,s)$ is differentiable with continuous partial derivatives and $\frac{\partial}{\partial t}k(t,s)=-\frac{\partial}{\partial s}k(t,s),$

(ii) $|\frac{\partial}{\partial t}k(t,s)|\leq c(t-s)^{\alpha-2}$ for some $\alpha\in (\frac{1}{2},1),$

(iii) $\int_s^t|k(t,u)|du\leq c(t-s)^{\bar{\alpha}}$ for some $\bar{\alpha}>\frac{1}{2},$

(iv) $\int_0^t|k(t,u)|^{p_0}du\leq c$ for some $p_0>1.$

\noindent Under $(A_1)$ and regular conditions on $b, \sigma$ and $g,$ we are able to obtain a nice It\^o differential representation for the solution in Theorem \ref{tmkl3}. This representation allows us to use It\^o differential formula for studying deeper properties of the solution. As a non-trivial application, our Theorem \ref{ttklms} provides an explicit bound on Wasserstein distance in the small time central limit theorem for singular stochastic Volterra integral equations of the form
\begin{equation}\label{fko3}
X_t=x_0+\int_0^t(t-s)^{\alpha-1}b(s,X_s)ds+\int_0^t\sigma(s,X_s)dB_s,\,\,\,t\in[0,T],
\end{equation}
where $x_0\in \mathbb{R}$ and $\alpha\in (\frac{1}{2},1).$ Here we set $k(t,s)=(t-s)^{\alpha-1}$ to illustrate the singularity of $k(t,s)$ and $g(t,s)=1$ for the simplicity. We also note that  the small time central limit theorem is useful particularly for studying the digital options in finance, see e.g. \cite{Gerhold2015}.

\section{The main results}
In the whole paper, we denote by $C$ a generic constant, whose value may change from one line to another. Since our investigation focuses on the singularity of $k(t,s),$ we are going to impose the following fundamental conditions on $b, \sigma$ and $g.$

\noindent$(A_2)$ The coefficients $b, \sigma: [0, T] \times \mathbb{R}\to \mathbb{R}$ are Lipschitz and have linear growth, i.e. there exists $L> 0$ such that: 
$$ |b(t,x)-b(t,y)|+|\sigma(t,x)-\sigma(t,y)| \leq L|x-y|, \quad \forall x, y\in \mathbb{R}, t\in[0,T] $$
and
$$ |b(t,x)|+|\sigma(t,x)| \leq L(1+|x|), \quad \forall x\in \mathbb{R}, t\in[0,T].$$
The kernel $g:[0,T]^2\to \mathbb{R}$ is differentiable in the variable $t,$ and both $g(t,s)$ and $\frac{\partial}{\partial t}g(t,s)$ are continuous.

\begin{prop}\label{kol,w}Let $(X_t)_{t\in [0,T]}$ be the solution of the equation (\ref{fko3k1}). Suppose Assumptions $(A_1)$ and $(A_2).$ Then, for any $p\geq1,$ there exists a positive constant $C$ such that
\begin{equation}\label{jdkm2}
E|X_t|^p\leq C,\,\,\,\forall\,t\in[0,T]
\end{equation}
and
\begin{equation}\label{jdkm3}
E|X_t-X_s|^p\leq C|t-s|^{\frac{p}{2}},\,\,\,\forall\,s,t\in[0,T].
\end{equation}
\end{prop}
\begin{proof}By Lyapunov's inequality, we only need to prove (\ref{jdkm2}) and (\ref{jdkm3}) for $p>\frac{p_0}{p_0-1}\vee 2.$ Applying H\"older inequality with $p>\frac{p_0}{p_0-1},q=\frac{p}{p-1}<p_0$  yields
\begin{align*}
\big|\int_0^tk(t,s)b(s,X_s)ds\big|&\leq \bigg(\int_0^t|k(t,s)|^qds\bigg)^{\frac{1}{q}}\bigg(\int_0^t|b(s,X_s)|^pds\bigg)^{\frac{1}{p}},\,\,\,\forall\,t\in[0,T].
\end{align*}
Hence, by Assumption $(A_1),(iv),$ there exists a positive constant $C$ such that
\begin{align}
E\big|\int_0^tk(t,s)b(s,X_s)ds\big|^p&\leq C\int_0^tE|b(s,X_s)|^pds,\,\,\,\forall\,t\in[0,T].\label{jij2}
\end{align}
By the Burkholder-Davis-Gundy inequality there exists a positive constant $c_p$ such that
\begin{align*}
E\big|\int_0^tg(t,s)\sigma(s,X_s)dB_s\big|^p&\leq c_pE\big|\int_0^tg^2(t,s)\sigma^2(s,X_s)ds\big|^{\frac{p}{2}},\,\,\,\forall\,t\in[0,T].
\end{align*}
Then, by H\"older inequality and the continuity of $g(t,s),$ for each $p>2,$
\begin{align}
E\big|\int_0^tg(t,s)\sigma(s,X_s)dB_s\big|^p&\leq C\int_0^tE|\sigma(s,X_s)|^pds,\,\,\,\forall\,t\in[0,T].\label{jij3}
\end{align}
For each $p>\frac{p_0}{p_0-1}\vee 2,$ we now use the fundamental inequality $(a+b+c)^p\leq 3^{p-1}(a^p+b^p+c^p)$ and linear growth property of $b$ and $\sigma$ to get
\begin{align*}
E|X_t|^p&\leq 3^{p-1}\bigg(|x_0|^p+E\big|\int_0^tk(t,s)b(s,X_s)ds\big|^p+E\big|\int_0^tg(t,s)\sigma(s,X_s)dB_s\big|^p\bigg)\\
&\leq C\bigg(|x_0|^p+\int_0^tE|b(s,X_s)|^pds+\int_0^tE|\sigma(s,X_s)|^pds\bigg)\\
&\leq C\bigg(|x_0|^p+\int_0^tL^p2^{p-1}(1+E|X_s|^p)ds+\int_0^tL^p2^{p-1}(1+E|X_s|^p)ds\bigg)\\
&\leq C\bigg(1+\int_0^tE|X_s|^pds\bigg),\,\,\,\forall\,t\in[0,T].
\end{align*}
So we can obtain (\ref{jdkm2}) by using Gronwall's lemma.

It only remains to verify (\ref{jdkm3}). Without loss of generality we assume that $0\leq s<t\leq T.$ We have
\begin{multline*}
X_t-X_s=\int_0^s[k(t,u)-k(s,u)]b(u,X_u)du+\int_s^tk(t,u)b(u,X_u)du\\+\int_0^s[g(t,u)-g(s,u)]\sigma(u,X_u)dB_u+\int_s^tg(t,u)\sigma(u,X_u)dB_u
\end{multline*}
and
\begin{multline}\label{omwn5}
E|X_t-X_s|^p\leq 4^{p-1}\bigg(E\big|\int_0^s[k(t,u)-k(s,u)]b(u,X_u)du\big|^p+E\big|\int_s^tk(t,u)b(u,X_u)du\big|^p\\
+E\big|\int_0^s[g(t,u)-g(s,u)]\sigma(u,X_u)dB_u\big|^p+E\big|\int_s^tg(t,s)\sigma(u,X_u)dB_u\big|^p\bigg).
\end{multline}
Notice that $E|b(u,X_u)|^p+E|\sigma(u,X_u)|^p\leq C\,\forall\,u\in[0,T].$ We use H\"older inequality to get
\begin{align*}
E\big|&\int_0^s[k(t,u)-k(s,u)]b(u,X_u)du\big|^p\\
&\leq \left(\int_0^s|k(t,u)-k(s,u)|du\right)^{p-1}\int_0^s|k(t,u)-k(s,u)|E|b(u,X_u)|^pdu\\
&\leq C\left(\int_0^s|k(t,u)-k(s,u)|du\right)^{p}\\
&\leq C\left(\int_0^s[(s-u)^{\alpha-1}-(t-u)^{\alpha-1}]du\right)^{p}\,\,\text{by Assumption $(A_1),(ii)$}\\
&\leq C\left(s^{\alpha}+(t-s)^{\alpha}-t^{\alpha}]\right)^{p}\\
&\leq C(t-s)^{p\alpha},\,\,\,0\leq s\leq t\leq T,
\end{align*}
and
\begin{align*}
E\big|\int_s^tk(t,u)b(u,X_u)du\big|^p&\leq \left(\int_s^t|k(t,u)|du\right)^{p-1}\int_s^t|k(t,u)|E|b(u,X_u)|^pdu\\
&\leq C\left(\int_s^t|k(t,u)|du\right)^{p}\\
&\leq C(t-s)^{p\bar{\alpha}}\,\,\text{by Assumption $(A_1),(iii)$}.
\end{align*}
Since  $\frac{\partial}{\partial t}g(t,s)$ is continuous in $[0,T]^2,$ this implies that $|g(t,u)-g(s,u)|\leq C|t-s|.$ Hence, by the Burkholder-Davis-Gundy inequality, we can obtain
\begin{align*}
E\big|\int_0^s[g(t,u)-g(s,u)]\sigma(u,X_u)dB_u\big|^p\leq C|t-s|^p,\,\,\,\forall\,s,t\in[0,T].
\end{align*}
Also, it is easy to see that $E\big|\int_s^tg(t,u)\sigma(u,X_u)dB_u\big|^p\leq C|t-s|^{\frac{p}{2}}.$ Hence, recalling (\ref{omwn5}), we obtain
$$E|X_t-X_s|^p\leq C(|t-s|^{p\alpha}+|t-s|^{p\bar{\alpha}}+|t-s|^p+|t-s|^{\frac{p}{2}}),\,\,\,\forall\,s,t\in[0,T],$$
which leads us to (\ref{jdkm3}) because $\alpha,\bar{\alpha}>\frac{1}{2}.$ This completes the proof.
\end{proof}
We now are in a position to state and prove the first main result of the paper. To handle the singularity of $k(t,s),$ we need to additionally impose the following condition.

\noindent$(A_3)$ Given $\alpha$ as in Assumption $(A_1),(ii).$ There exists $\beta_1>1-\alpha$ such that
$$ |b(t,x)-b(s,x)|\leq L|t-s|^{\beta_1}, \quad \forall x\in \mathbb{R}, t,s\in[0,T] $$

\begin{thm}\label{tmkl3} Suppose Assumptions  $(A_1),(A_2)$ and $(A_3).$ The solution $(X_t)_{t\in[0,T]}$ of the equation (\ref{fko3k1}) admits the following It\^o differential representation
\begin{equation}\label{ymke}
dX_t=\big(k(t,0)b(t,X_t)+\varphi(t)+G(t)\big)dt+g(t,t)\sigma(t,X_t)dB_t,\,\,\,t\in[0,T],
\end{equation}
where $X_0=x_0$ and
$$\varphi(t):=-\int_0^t\frac{\partial}{\partial t}k(t,u)[b(t,X_t)-b(u,X_u)]du,\,\,\,G(t):=\int_0^t\frac{\partial}{\partial t}g(t,u)\sigma(u,X_u)dB_u\,\,\,t\in[0,T].$$
\end{thm}
\begin{proof}We separate the proof into four steps.

\noindent{\it Step 1.} We first verify that $\varphi(s),s\in[0,T]$ is well defined. By Kolmogorov continuity theorem, it follows from (\ref{jdkm3}) that for any $\delta\in(0,\frac{1}{2})$ there exists a finite random variable $C(\omega)$ such that
$$|X_t-X_s|\leq C(\omega)|t-s|^{\frac{1}{2}-\delta},\,\,\,\forall\,s,t\in[0,T].$$
By Lipschitz property of $b$ and Assumption $(A_3)$ we deduce
\begin{equation}\label{jdkm33}
|b(s,X_s)-b(u,X_u)|\leq |b(s,X_s)-b(u,X_s)|+|b(u,X_s)-b(u,X_u)|\leq L[(s-u)^{\beta_1}+|X_s-X_u|]
\end{equation}
for all $s,u\in[0,T].$
We choose $\delta=\frac{\alpha}{2}-\frac{1}{4}\in(0,\frac{1}{2})$ to get
\begin{align*}
|\varphi(s)|
&\leq L\int_0^s\big|\frac{\partial}{\partial s}k(s,u)\big|[(s-u)^{\beta_1}+|X_s-X_u|]du\\
&\leq Lc\int_0^s(s-u)^{\alpha-2}[(s-u)^{\beta_1}+C(\omega)|t-s|^{\frac{1}{2}-\delta}]du\\
&=Lc\bigg[\frac{s^{\alpha+\beta_1-1}}{\alpha+\beta_1-1}+C(\omega)\frac{s^{\frac{\alpha}{2}-\frac{1}{4}}}{\frac{\alpha}{2}-\frac{1}{4}}\bigg],\,\,\,s\in[0,T].
\end{align*}
So $\varphi(s),s\in[0,T]$ is well defined because $\beta_1>1-\alpha$  and $\alpha>\frac{1}{2}.$ Furthermore, we have
\begin{align}
E|\varphi(s)|&\leq L\int_0^s\big|\frac{\partial}{\partial s}k(s,u)\big|[(s-u)^{\beta_1}+E|X_s-X_u|]du\notag\\
&\leq CL\int_0^s(s-u)^{\alpha-2}[(s-u)^{\beta_1}+(s-u)^{\frac{1}{2}}]du\notag\\
&\leq CL\bigg[\frac{s^{\alpha+\beta_1-1}}{\alpha+\beta_1-1}+\frac{s^{\alpha-\frac{1}{2}}}{\alpha-\frac{1}{2}}\bigg],\,\,\,s\in[0,T].\label{unkm1}
\end{align}
\noindent{\it Step 2.} For each $\varepsilon\in(0,1),$ we consider the stochastic Volterra integral equation
\begin{equation}\label{fko4}
X^{(\varepsilon)}_t=x_0+\int_0^tk(t+\varepsilon,s)b(s,X^{(\varepsilon)}_s)ds+\int_0^tg(t,s)\sigma(s,X^{(\varepsilon)}_s)dB_s,\,\,\,t\in[0,T].
\end{equation}
Since the kernel $k(t+\varepsilon,s)$ is smooth, Assumption $(A_2)$ ensures that the equation (\ref{fko4}) admits a unique solution $(X^{(\varepsilon)}_t)_{t\in[0,T]}.$ By using the same arguments as in the proof of Proposition \ref{kol,w} we also obtain
\begin{equation}\label{jdm3d}
E|X^{(\varepsilon)}_t-X^{(\varepsilon)}_s|^p\leq C|t-s|^{\frac{p}{2}},\,\,\,\forall\,s,t\in[0,T].
\end{equation}
We claim that, for any $p\geq 1,$
\begin{equation}\label{y7m9}
E|X^{(\varepsilon)}_t-X_t|^p\leq C\varepsilon^{p\alpha},\,\,\,\forall\,t\in[0,T].
\end{equation}
Here we emphasize that the constants $C$ in (\ref{jdm3d}) and (\ref{y7m9}) do not depend on $\varepsilon\in(0,1).$ We have
\begin{align*}
&X^{(\varepsilon)}_t-X_t=\int_0^t[k(t+\varepsilon,s)b(s,X^{(\varepsilon)}_s)-k(t,s)b(s,X_s)]ds+\int_0^tg(t,s)
[\sigma(s,X^{(\varepsilon)}_s)-\sigma(s,X_s)]dB_s\\
&=\int_0^t[k(t+\varepsilon,s)-k(t,s)]b(s,X_s)ds+\int_0^tk(t+\varepsilon,s)[b(s,X^{(\varepsilon)}_s)-b(s,X_s)]ds\\
&+\int_0^tg(t,s)[\sigma(s,X^{(\varepsilon)}_s)-\sigma(s,X_s)]dB_s,\,\,\,t\in[0,T].
\end{align*}
We use H\"older inequality to get, for any $p>1,$
\begin{align}
E\big|\int_0^t&[k(t+\varepsilon,s)-k(t,s)]b(s,X_s)ds\big|^p\notag\\
&\leq \bigg(\int_0^t|k(t+\varepsilon,s)-k(t,s)|ds\bigg)^{p-1}\int_0^t|k(t+\varepsilon,s)-k(t,s)|E|b(s,X_s)|^pds\notag\\
&\leq C\bigg(\int_0^t|k(t+\varepsilon,s)-k(t,s)|ds\bigg)^{p}\notag\\
&\leq C\bigg(\int_0^t[(t-s)^{\alpha-1}-(t-s+\varepsilon)^{\alpha-1}]ds\bigg)^p\,\,\text{by Assumption $(A_1),(ii)$}\notag\\
&\leq C\big(t^{\alpha}-(t+\varepsilon)^{\alpha}+\varepsilon^{\alpha}\big)^p\leq C\varepsilon^{p\alpha},\,\,\,t\in[0,T].
\end{align}
By using the same arguments as in the proof (\ref{jij2}) we obtain, for any $p> \frac{p_0}{p_0-1},$
\begin{align}
E\big|\int_0^tk(t+\varepsilon,s)[b(s,X^{(\varepsilon)}_s)-b(s,X_s)]ds\big|^p
&\leq C\int_0^t E|b(s,X^{(\varepsilon)}_s)-b(s,X_s)|^pds\notag\\
&\leq CL^p\int_0^tE|X^{(\varepsilon)}_s-X_s|^pds,\,\,\,t\in[0,T].
\end{align}
By using the same arguments as in the proof (\ref{jij3}) we obtain, for any $p> 2,$
\begin{align}
E\big|\int_0^tg(t,s)[\sigma(s,X^{(\varepsilon)}_s)-\sigma(s,X_s)]dB_s\big|^p
&\leq C\int_0^t E|\sigma(s,X^{(\varepsilon)}_s)-\sigma(s,X_s)|^pds\notag\\
&\leq CL^p\int_0^tE|X^{(\varepsilon)}_s-X_s|^pds,\,\,\,t\in[0,T].\label{y7m9a}
\end{align}
So, for each $p>\frac{p_0}{p_0-1}\vee 2,$ we can get
$$E|X^{(\varepsilon)}_t-X_t|^p\leq C\varepsilon^{p\alpha}+C\int_0^tE|X^{(\varepsilon)}_s-X_s|^pds,\,\,\,t\in[0,T].$$
By Gronwall's lemma $E|X^{(\varepsilon)}_t-X_t|^p\leq C\varepsilon^{p\alpha}e^{Ct}\leq C\varepsilon^{p\alpha}$ for all $t\in[0,T].$ Thus, by Lyapunov's inequality, the claim (\ref{y7m9}) is verified for  any $p\geq 1.$

\noindent{\it Step 3.} In this step, we represent (\ref{fko4}) as an It\^o stochastic differential equation. We put
$$Y^{(\varepsilon)}_t:=\int_0^tk(t+\varepsilon,s)b(s,X^{(\varepsilon)}_s)ds,\,\,\,t\in[0,T],$$
$$\varphi^{(\varepsilon)}(s):=-\int_0^s\frac{\partial}{\partial s}k(s+\varepsilon,u)[b(s,X^{(\varepsilon)}_s)-b(u,X^{(\varepsilon)}_u)]du,\,\,\,s\in[0,T].$$
We have
\begin{align*}
\frac{dY^{(\varepsilon)}_t}{dt}&=k(t+\varepsilon,t)b(t,X^{(\varepsilon)}_t)+\int_0^t\frac{\partial}{\partial t}k(t+\varepsilon,s)b(s,X^{(\varepsilon)}_s)ds\\
&=k(t+\varepsilon,t)b(t,X^{(\varepsilon)}_t)+\int_0^t\frac{\partial}{\partial t}k(t+\varepsilon,s)b(t,X^{(\varepsilon)}_t)ds-\int_0^t\frac{\partial}{\partial t}k(t+\varepsilon,s)[b(t,X^{(\varepsilon)}_t)-b(s,X^{(\varepsilon)}_s)]ds\\
&=k(t+\varepsilon,0)b(t,X^{(\varepsilon)}_t)-\int_0^t\frac{\partial}{\partial t}k(t+\varepsilon,s)[b(t,X^{(\varepsilon)}_t)-b(s,X^{(\varepsilon)}_s)]ds\,\,\text{by Assumption $(A_1),(i)$}\\
&=k(t+\varepsilon,0)b(t,X^{(\varepsilon)}_t)+\varphi^{(\varepsilon)}(t),
\end{align*}
or equivalently
\begin{align*}
\int_0^tk(t+\varepsilon,s)b(s,X^{(\varepsilon)}_s)ds&=\int_0^t\big(k(s+\varepsilon,0)b(s,X^{(\varepsilon)}_s)+\varphi^{(\varepsilon)}(s)\big)ds,
\,\,\,t\in[0,T]
\end{align*}
Inserting this relation into (\ref{fko4}) gives us
\begin{equation}\label{ujm1}
X^{(\varepsilon)}_t=x_0+\int_0^t\big(k(s+\varepsilon,0)b(s,X^{(\varepsilon)}_s)+\varphi^{(\varepsilon)}(s)\big)ds+\int_0^tg(t,s)\sigma(s,X^{(\varepsilon)}_s)
dB_s,\,\,\,t\in[0,T].
\end{equation}
\noindent{\it Step 4.} This step concludes the proof by letting $\varepsilon\to0^+.$ We first show that, for all $t\in[0,T],$
\begin{equation}\label{jimasd}
\int_0^t\varphi^{(\varepsilon)}(s)ds\to \int_0^t\varphi(s)ds\,\,\,\text{in}\,\,\,L^1(\Omega).
\end{equation}
We write $\varphi^{(\varepsilon)}(s)=\varphi^{(\varepsilon)}_1(s)+\varphi^{(\varepsilon)}_2(s),$ where
$$\varphi^{(\varepsilon)}_1(s):=-\int_0^s\frac{\partial}{\partial s}k(s+\varepsilon,u)[b(s,X^{(\varepsilon)}_s)-b(s,X_s)-b(u,X^{(\varepsilon)}_u)+b(u,X_u)]du,$$
$$\varphi^{(\varepsilon)}_2(s):=-\int_0^s\frac{\partial}{\partial s}k(s+\varepsilon,u)[b(s,X_s)-b(u,X_u)]du,\,\,\,s\in[0,T].$$
By Lipschitz property of $b$ and the claim (\ref{y7m9}) we deduce
\begin{align*}
E|b(s,X^{(\varepsilon)}_s)&-b(s,X_s)-b(u,X^{(\varepsilon)}_u)+b(u,X_u)|\\
&\leq E|b(s,X^{(\varepsilon)}_s)-b(s,X_s)|+ E|b(u,X^{(\varepsilon)}_u)-b(u,X_u)|\\
&\leq C \varepsilon^{\alpha},\,\,\,\forall\,s,u\in[0,T].
\end{align*}
On the other hand, we deduce from the estimates (\ref{jdkm3}), (\ref{jdkm33}) and (\ref{jdm3d}) that
\begin{align*}
E|b(s,X^{(\varepsilon)}_s)&-b(s,X_s)-b(u,X^{(\varepsilon)}_u)+b(u,X_u)|\\
&\leq E|b(s,X^{(\varepsilon)}_s)-b(u,X^{(\varepsilon)}_u)|+ E|b(s,X_s)-b(u,X_u)|\\
&\leq C [(s-u)^{\beta_1}+(s-u)^{\frac{1}{2}}],\,\,\,\forall\,s,u\in[0,T].
\end{align*}
Now for any $0<\delta<\min\{\frac{\alpha+\beta_1-1}{\beta_1},2\alpha-1\}$ we have
\begin{align*}
E|\varphi^{(\varepsilon)}_1(s)|&\leq C\int_0^s\big|\frac{\partial}{\partial s}k(s+\varepsilon,u)\big|\varepsilon^{\delta\alpha}[(s-u)^{\beta_1}+(s-u)^{\frac{1}{2}}]^{1-\delta}du\\
&\leq C\int_0^s(s-u+\varepsilon)^{\alpha-2}\varepsilon^{\delta\alpha}[(s-u)^{(1-\delta)\beta_1}+(s-u)^{\frac{1-\delta}{2}}]du\\
&\leq C\int_0^s(s-u)^{\alpha-2}\varepsilon^{\delta\alpha}[(s-u)^{(1-\delta)\beta_1}+(s-u)^{\frac{1-\delta}{2}}]du\\
&=C\varepsilon^{\delta\alpha}\bigg[\frac{s^{\alpha+\beta_1-1-\delta\beta_1}}{\alpha+\beta_1-1-\delta\beta_1}
+\frac{s^{\alpha-\frac{1}{2}-\frac{\delta}{2}}}{\alpha-\frac{1}{2}-\frac{\delta}{2}}\bigg],\,\,\,\forall\,s\in[0,T].
\end{align*}
Consequently,
$$\int_0^t\varphi^{(\varepsilon)}_1(s)ds\to 0\,\,\,\text{in}\,\,\,L^1(\Omega).$$
Once again, we use Lipschitz property of $b$ and the estimates (\ref{jdkm3}) and (\ref{jdkm33}) to obtain
\begin{align*}
E|\varphi^{(\varepsilon)}_2(s)-\varphi(s)|&\leq\int_0^s\big|\frac{\partial}{\partial s}k(s+\varepsilon,u)-\frac{\partial}{\partial s}k(s,u)\big|E|b(s,X_s)-b(u,X_u)|du\\
&\leq CL\int_0^s\big|\frac{\partial}{\partial s}k(s+\varepsilon,u)-\frac{\partial}{\partial s}k(s,u)\big| [(s-u)^{\beta_1}+(s-u)^{\frac{1}{2}}]du,\,\,\,s\in[0,T].
\end{align*}
This, together with the continuity of $\frac{\partial}{\partial s}k(s,u),$ implies $E|\varphi^{(\varepsilon)}_2(s)-\varphi(s)|\to 0$ as $\varepsilon\to 0^+.$ Moreover, it follows from Assumption $(A_1),(ii)$ that
\begin{align*}
E|\varphi^{(\varepsilon)}_2(s)-\varphi(s)|&\leq 2CLc\int_0^s(s-u)^{\alpha-2} [(s-u)^{\beta_1}+(s-u)^{\frac{1}{2}}]du,
\end{align*}
which is an integrable function on $[0,T]$ (see the estimate (\ref{unkm1})). Hence, by the dominated convergence theorem, we have
$$\int_0^t\varphi^{(\varepsilon)}_2(s)ds\to \int_0^t\varphi(s)ds\,\,\,\text{in}\,\,\,L^1(\Omega).$$
So the convergence (\ref{jimasd}) holds true. Similarly, we write $\int_0^tk(s+\varepsilon,0)b(s,X^{(\varepsilon)}_s)ds=\int_0^tk(s+\varepsilon,0)[b(s,X^{(\varepsilon)}_s)-b(s,X_s)]ds
+\int_0^tk(s+\varepsilon,0)b(s,X_s)ds$ and we can infer that, for all $t\in[0,T],$
\begin{equation}\label{jimasd2}
\int_0^tk(s+\varepsilon,0)b(s,X^{(\varepsilon)}_s)ds\to \int_0^t k(s,0)b(s,X_s)ds\,\,\,\text{in}\,\,\,L^1(\Omega).
\end{equation}
By (\ref{y7m9}) and (\ref{y7m9a}) we also have, for all $t\in[0,T],$
\begin{equation}\label{jimasd3}
\int_0^tg(t,s)\sigma(s,X^{(\varepsilon)}_s)dB_s\to \int_0^tg(t,s)\sigma(s,X_s)dB_s\,\,\,\text{in}\,\,\,L^1(\Omega).
\end{equation}
Recalling (\ref{ujm1}), we obtain from (\ref{jimasd}), (\ref{jimasd2}) and (\ref{jimasd3}) that
\begin{equation}\label{ymke1}
X_t=x_0+\int_0^t\big(k(s,0)b(s,X_s)+\varphi(s)\big)ds+\int_0^tg(t,s)\sigma(s,X_s)dB_s,\,\,\,t\in[0,T].
\end{equation}
To finish the proof, we use stochastic Fubini's theorem to get
\begin{align*}
\int_0^tG(s)ds&=\int_0^t\bigg(\int_0^s\frac{\partial}{\partial s}g(s,u)\sigma(u,X_u)dB_u\bigg)ds\\
&=\int_0^t\bigg(\int_u^t\frac{\partial}{\partial s}g(s,u)\sigma(u,X_u)ds\bigg)dB_u\\
&=\int_0^tg(t,u)\sigma(u,X_u)dB_u-\int_0^tg(u,u)\sigma(u,X_u)dB_u.
\end{align*}
As a consequence, (\ref{ymke}) follows from (\ref{ymke1}).
\end{proof}
Let us now consider the equation (\ref{fko3}). It is easy to see that $k(t,s)=(t-s)^{\alpha-1}$ satisfies Assumption $(A_1).$ Hence, under Assumptions $(A_2)$ and $(A_3),$ we can rewrite (\ref{fko3}) as follows
\begin{equation}\label{yy7e}
X_t=x_0+\int_0^t \big(s^{\alpha-1}b(s,X_s)+\varphi(s)\big)ds+\int_0^t\sigma(s,X_s)dB_s,\,\,\,t\in[0,T],
\end{equation}
The results obtained in \cite{Gerhold2015}  tell us that $X_t$ fulfills the small time central limit theorem, i.e. $\frac{X_a-x_0}{\sqrt{a}}\to N_\sigma$ in distribution as $a\to0^+,$ where $N_{\sigma}$ is a normal random variable with mean $0$ and variance $\sigma^2(0,x_0).$ Our purpose here is to go a further step, we would like to investigate the rate of this convergence. For this purpose, we will provide an explicit bound on Wasserstein distance. Recall that the Wasserstein distance between the laws of two random variables $F$ and $G$ is defined by
$$d_W(F,G):=\sup\limits_{|h(x)-h(y)|\leq |x-y|}|E[h(F)]-E[h(G)]|=\sup\limits_{h\in \mathcal{C}^1,\|h'\|_\infty\leq 1}|E[h(F)]-E[h(G)]|,$$
where $\|.\|_\infty$ denotes the supremum norm, see e.g. \cite{Chen2015}.

The next statement is the second main result of the present paper.
\begin{thm}\label{ttklms}Suppose Assumptions  $(A_2)$ and $(A_3).$ In addition, we assume that there exists $\beta_2>0$ such that
\begin{equation}\label{tyn3}
|\sigma(t,x)-\sigma(0,x)|\leq L|t|^{\beta_2}, \quad \forall x\in \mathbb{R}, t\in[0,T] .
\end{equation}
Then, the solution $(X_t)_{t\in[0,T]}$ of the equation (\ref{fko3}) satisfies
$$d_W\left(\frac{X_a-x_0}{\sqrt{a}},N_{\sigma}\right)\leq C \left( a^{\alpha-\frac{1}{2}}+ a^{\frac{\beta_2}{2}\wedge\frac{1}{4}}\right),\,\,\,a\to 0^+,$$
where $C$ is a positive constant not depending on $a.$
\end{thm}
\begin{proof}We separate the proof into two steps.

\noindent{\it Step 1.} Denote by $\mathcal{C}^2_b$ the space of twice differentiable functions with bounded derivatives. In this step, we claim that for any $h\in \mathcal{C}^2_b$ and $a\in (0,1],$ we have
\begin{equation}\label{d4is1t3}
\big|Eh\left(\frac{X_a-x_0}{\sqrt{a}}\right)-Eh(N_\sigma)\big|\leq C \left(\|h'\|_\infty a^{\alpha-\frac{1}{2}}+\|h''\|_\infty a^{\beta_2\wedge\frac{1}{2}}\right),
\end{equation}
where $C$ is a positive constant not depending on $a$ and $h.$

We consider the interpolation function $\mathcal{H}:[0,a]\times \mathbb{R}\longrightarrow \mathbb{R}$ which is defined by
 $$\mathcal{H}(t,x)=\int_{-\infty}^\infty h\left(\frac{x-x_0}{\sqrt{a}}+\sigma(0,x_0)\sqrt{1-\frac{t}{a}}\,\,z\right)\phi(z)dz,$$
where $\phi(z)=\frac{1}{\sqrt{2\pi}}e^{-\frac{z^2}{2}}$ is the density function of standard normal random variable. Let $Z$ denote a standard normal random variable that is independent of $B.$ We have
\begin{align*}
E\left[h\left(\frac{X_t-x_0}{\sqrt{a}}+\sigma(0,x_0)\sqrt{1-\frac{t}{a}}Z\right)\right]
&=E\left[E\left[h\left(\frac{x-x_0}{\sqrt{a}}+\sigma(0,x_0)\sqrt{1-\frac{t}{a}}Z\right)\right]\bigg|_{x=X_t}\right]\\
&=E[\mathcal{H}(t,X_t)],\,\,\,0\leq t\leq a.
\end{align*}
Hence, we can obtain the following relations
\begin{equation}\label{dmsow2}
\text{$Eh\left(\frac{X_a-x_0}{\sqrt{a}}\right)=E\mathcal{H}(a,X_a)$ and $Eh(N_{\sigma})=E\mathcal{H}(0,x_0)$}.
\end{equation}
By straightforward calculations we obtain
$$\frac{\partial}{\partial x}\mathcal{H}(t,x)=\frac{1}{\sqrt{a}}\int_{-\infty}^\infty h'\left(\frac{x-x_0}{\sqrt{a}}+\sigma(0,x_0)\sqrt{1-\frac{t}{a}}\,\,z\right)\phi(z)dz,$$
$$\frac{\partial^2}{\partial x^2}\mathcal{H}(t,x)=\frac{1}{a}\int_{-\infty}^\infty h''\left(\frac{x-x_0}{\sqrt{a}}+\sigma(0,x_0)\sqrt{1-\frac{t}{a}}\,\,z\right)\phi(z)dz,$$
As a consequence,
\begin{equation}\label{dkds}
\big|\frac{\partial}{\partial x}\mathcal{H}(t,x)\big|\leq \frac{\|h'\|_\infty }{\sqrt{a}},\,\,\,\,\big|\frac{\partial^2}{\partial y^2}\mathcal{H}(t,x)\big|\leq \frac{\|h''\|_\infty }{a}
\end{equation}
for all $(t,x)\in [0,a]\times \mathbb{R}.$ On the other hand, we have
\begin{align*}\frac{\partial}{\partial t}\mathcal{H}(t,x)&=\frac{-\sigma(0,x_0)}{2a\sqrt{1-\frac{t}{a}}}\int_{-\infty}^\infty h'\left(\frac{x-x_0}{\sqrt{a}}+\sigma(0,x_0)\sqrt{1-\frac{t}{a}}\,\,z\right)z\phi(z)dz\\
&=\frac{\sigma(0,x_0)}{2a\sqrt{1-\frac{t}{a}}}\int_{-\infty}^\infty h'\left(\frac{x-x_0}{\sqrt{a}}+\sigma(0,x_0)\sqrt{1-\frac{t}{a}}\,\,z\right)d\phi(z),
\end{align*}
here we note that  $\phi'(z)=-z\phi(z).$ Then, by the integration by parts formula
\begin{align}\frac{\partial}{\partial t}\mathcal{H}(t,x)&=-\frac{\sigma^2(0,x_0)}{2a}\int_{-\infty}^\infty h''\left(\frac{x-x_0}{\sqrt{a}}+\sigma(0,x_0)\sqrt{1-\frac{t}{a}}\,\,z\right)\phi(z)dz\nonumber\\
&=-\frac{\sigma^2(0,x_0)}{2}\frac{\partial^2}{\partial x^2}\mathcal{H}(t,x).\label{thd10}
\end{align}
By It\^o differential formula, we obtain from (\ref{yy7e}) that
\begin{align*}
\mathcal{H}(t,X_t)-\mathcal{H}(0,x_0)&=\int_0^t \frac{\partial}{\partial s}\mathcal{H}(s,X_s)ds+\int_0^t\frac{\partial}{\partial x}\mathcal{H}(s,X_s)dX_s+\frac{1}{2}\int_0^t\frac{\partial^2}{\partial x^2}\mathcal{H}(s,X_s)\sigma^2(s,X_s)ds\\
&=\int_0^t \bigg(\frac{\partial}{\partial s}\mathcal{H}(s,X_s)+\frac{\partial}{\partial x}\mathcal{H}(s,X_s)\bar{b}(s,X_s)+\frac{1}{2}\frac{\partial^2}{\partial x^2}\mathcal{H}(s,X_s)\sigma^2(s,X_s)\bigg)ds\\
&
+\int_0^t \frac{\partial}{\partial x}\mathcal{H}(s,X_s)\sigma(s,X_s)dB_s,\,\,\,t\in[0,a],
\end{align*}
where $\bar{b}(s,X_s):=s^{\alpha-1}b(s,X_s)+\varphi(s).$ Then, by the relation (\ref{thd10})
\begin{multline}\label{clts04}
\mathcal{H}(t,X_t)-\mathcal{H}(0,x_0)=\int_0^t \bigg(\frac{\partial}{\partial x}\mathcal{H}(s,X_s)\bar{b}(s,X_s)+\frac{1}{2}\frac{\partial^2}{\partial x^2}\mathcal{H}(s,X_s)[\sigma^2(s,X_s)-\sigma^2(0,x_0)]\bigg)ds\\
+\int_0^t \frac{\partial}{\partial x}\mathcal{H}(s,X_s)\sigma(s,X_s)dB_s,\,\,\,t\in[0,a].
\end{multline}
We therefore obtain
\begin{equation}\label{o1}
E[\mathcal{H}(a,X_a)]-E[\mathcal{H}(0,x_0)]=\int_0^a E\bigg[\frac{\partial}{\partial x}\mathcal{H}(s,X_s)\bar{b}(s,X_s)+\frac{1}{2}\frac{\partial^2}{\partial x^2}\mathcal{H}(s,X_s)[\sigma^2(s,X_s)-\sigma^2(0,x_0)]\bigg]ds.
\end{equation}
It follows from the estimates (\ref{unkm1}) and (\ref{dkds}) that
\begin{align}
\big|\int_0^a E\bigg[\frac{\partial}{\partial x}\mathcal{H}(s,X_s)\bar{b}(s,X_s)\bigg]ds\big|&\leq \frac{\|h'\|_\infty }{\sqrt{a}}\int_0^a E|\bar{b}(s,X_s)|ds\notag\\
&\leq \frac{\|h'\|_\infty}{\sqrt{a}}\int_0^a[s^{\alpha-1}E|b(s,X_s)|+E|\varphi(s)|]ds\notag\\
&\leq \frac{\|h'\|_\infty }{\sqrt{a}}\int_0^a\bigg[Cs^{\alpha-1}+CL\bigg(\frac{s^{\alpha+\beta_1-1}}{\alpha+\beta_1-1}
+\frac{s^{\alpha-\frac{1}{2}}}{\alpha-\frac{1}{2}}\bigg)\bigg]ds\notag\\
&\leq \frac{\|h'\|_\infty}{\sqrt{a}}\bigg[C\frac{a^{\alpha}}{\alpha}+CL\bigg(\frac{a^{\alpha+\beta_1}}{(\alpha+\beta_1)(\alpha+\beta_1-1)}
+\frac{a^{\alpha+\frac{1}{2}}}{\alpha^2-\frac{1}{4}}\bigg)\bigg]\notag\\
&\leq C\|h'\|_\infty a^{\alpha-\frac{1}{2}},\,\,\,a\in (0,1].\label{o2}
\end{align}
From (\ref{jdkm3}) and (\ref{tyn3}) we obtain
\begin{align*}
E|\sigma^2(s,X_s)-\sigma^2(0,x_0)|&\leq (E|\sigma(s,X_s)+\sigma(0,x_0)|^2)^{\frac{1}{2}}(E|\sigma(s,X_s)-\sigma(0,x_0)|^2)^{\frac{1}{2}}\\
&\leq C[(E|\sigma(s,X_s)-\sigma(0,X_s)|^2)^{\frac{1}{2}}+(E|\sigma(0,X_s)-\sigma(0,x_0)|^2)^{\frac{1}{2}}]\\
&\leq CL(s^{\beta_2}+s^{\frac{1}{2}}).
\end{align*}
Hence
\begin{align}
\big|\int_0^a E\bigg[\frac{1}{2}\frac{\partial^2}{\partial x^2}\mathcal{H}(s,X_s)[\sigma^2(s,X_s)-\sigma^2(0,x_0)]\bigg]ds\big|&\leq\frac{\|h''\|_\infty}{2a} \int_0^a E|\sigma^2(s,X_s)-\sigma^2(0,x_0)|ds\notag\\
&\leq\frac{\|h''\|_\infty}{2a} CL\bigg(\frac{a^{\beta_2+1}}{\beta_2+1}+\frac{2a^{\frac{3}{2}}}{3}\bigg)\notag\\
&\leq C\|h''\|_\infty a^{\beta_2\wedge \frac{1}{2}},\,\,\,a\in (0,1].\label{o3}
\end{align}
So we obtain (\ref{d4is1t3}) by combining (\ref{dmsow2}), (\ref{o1}), (\ref{o2}) and (\ref{o3}).

\noindent{\it Step 2.} In this step we use the smoothing technique to bound Wasserstein distance from (\ref{d4is1t3}). Let $h\in \mathcal{C}^1$ with $\|h'\|_\infty\leq 1.$ For each $u \in(0,1),$ define the function
$$h_u(x):=\int_{-\infty}^\infty h(\sqrt{u}y+\sqrt{1-u}\,x)\phi(y)dy,\,\,\,x\in \mathbb{R}.$$
Obviously, $\|h'_u\|_\infty\leq\sqrt{1-u}\|h'\|_\infty\leq 1.$ Moreover, as in the proof of Theorem 3.6 in \cite{Privault2015}, we have
$$\|h''_u\|_\infty\leq \frac{1}{\sqrt{u}},\,\,\,|E[h(F)]-E[h_u(F)]|\leq \sqrt{u}\left(1+\frac{E|F|}{2}\right)$$
for any $u\in(0,\frac{1}{2})$ and $F\in L^2(\Omega).$ Those estimates, combined with the result of {\it Step 1} yield
\begin{align*}
&\big|Eh\left(\frac{X_a-x_0}{\sqrt{a}}\right)-Eh(N_\sigma)\big|\\
&\leq \big|Eh\left(\frac{X_a-x_0}{\sqrt{a}}\right)-Eh_u\left(\frac{X_a-x_0}{\sqrt{a}}\right)\big|+\left|Eh(N_\sigma)-Eh_u(N_\sigma)\right|
+\big|Eh_u\left(\frac{X_a-x_0}{\sqrt{a}}\right)-Eh_u(N_\sigma)\big|\\
&\leq  \sqrt{u}\left(1+\frac{E|X_a-x_0|}{2\sqrt{a}}\right)+ \sqrt{u}\left(1+\frac{E|N_\sigma|}{2}\right)+C \left(\|h'_u\|_\infty a^{\alpha-\frac{1}{2}}+\|h''_u\|_\infty a^{\beta_2\wedge\frac{1}{2}}\right)\\
&\leq  \sqrt{u}\left(1+\frac{E|X_a-x_0|}{2\sqrt{a}}\right)+ \sqrt{u}\left(1+\frac{|\sigma(0,x_0)|}{2}\right)+C \left(a^{\alpha-\frac{1}{2}}+\frac{1}{\sqrt{u}} a^{\beta_2\wedge\frac{1}{2}}\right),\,\,u\in(0,\frac{1}{2}).
\end{align*}
We observe from (\ref{jdkm3})  that $\frac{E|X_a-x_0|}{\sqrt{a}}\leq C,a\in (0,1].$ Hence,
\begin{align}
\big|Eh\left(\frac{X_a-x_0}{\sqrt{a}}\right)-Eh(N_\sigma)\big|&\leq \frac{ \sqrt{u}}{2}\left(4+C+|\sigma(0,x_0)|\right)+\frac{C}{\sqrt{u}}a^{\beta_2\wedge\frac{1}{2}}+Ca^{\alpha-\frac{1}{2}},\,\,u\in(0,\frac{1}{2}).\label{hj2dk01}
\end{align}
Minimizing in $u,$ the right hand side of (\ref{hj2dk01}) attains its minimum value at $u_0:=\frac{2Ca^{\beta_2\wedge\frac{1}{2}}}{4+C+|\sigma(0,x_0)|}.$
When $a\to 0^+,$ we have $u_0\in (0,\frac{1}{2}).$ So the conclusion follows substituting $u_0$ in (\ref{hj2dk01}) and then taking the supremum over all $h\in \mathcal{C}^1$ with $\|h'\|_\infty\leq 1.$

The proof of Theorem is complete.
\end{proof}

\noindent {\bf Acknowledgments.} This research was funded by Vietnam National Foundation for Science and Technology Development (NAFOSTED) under grant number 101.03-2019.08.

\end{document}